\documentclass[11pt]{article}
\usepackage{}
\usepackage[total={6in, 9in}]{geometry}
\usepackage{amsmath,amsthm,amssymb}
\usepackage{mathrsfs}
\usepackage{amsfonts, color}
\usepackage{latexsym}
\usepackage{enumerate}
\usepackage{dsfont}
\usepackage{cases}
\usepackage{circuitikz}
\usepackage{amsmath}
\usepackage[
pdfauthor={},
pdftitle={},
pdfstartview=XYZ,
bookmarks=true,
colorlinks=true,
linkcolor=blue,
urlcolor=blue,
citecolor=blue,
bookmarks=true,
linktocpage=true,
hyperindex=true
]{hyperref}

\usepackage{graphicx}

\theoremstyle{plain}
\newtheorem{thm}{Theorem}[section]

\newtheorem{lem}[thm]{Lemma}

\newtheorem{coro}[thm]{Corollary}

\newtheorem{claim}[thm]{Claim}

\usepackage{indentfirst}

\theoremstyle{plain}

\theoremstyle{plain}

\theoremstyle{plain}

\title{A sufficient condition for a hypergraph to have a Berge-$k$-factor}
\author{Yuping Gao, Songling Shan, Gexin Yu\\
{\small a. School of Mathematics and Statistics, Lanzhou University, Lanzhou 730000, China}\\
{\small b. Department of Mathematics and Statistics, Auburn  University, Auburn, AL 36849, USA}\\
{\small c. Department of Mathematics, William \& Mary, Williamsburg, VA 23185, USA}}
\date{}

\begin{document}

\maketitle
\emph{\textbf{Abstract}.}
For any graph (hypergraph) $G$ with vertex set $V$ and edge set $E$, we define its incidence bipartite graph $\mathcal{I}(G)$ as the bipartite graph with bipartition $(E, V)$, where an edge $e \in E$ is adjacent to a vertex $v \in V$ in $\mathcal{I}(G)$ if and only if $e$ is incident to $v$ in $G$. This representation allows all concepts and properties of $G$ to be reformulated in terms of those of $\mathcal{I}(G)$.
In this paper, we investigate the notions of graph toughness and $k$-factors in bipartite graphs through this incidence perspective. As an application, our result  implies  the classic  theorem of Enomoto, Jackson, Katerinis, and Saito: for any integer $k \geq 1$, a $k$-tough graph $G$ has a $k$-factor if $k |V(G)|$ is even and $|V(G)| \geq k+1$. Furthermore, we extend this result to hypergraphs, without requiring uniformity.

\medskip


\section{Introduction}

All graphs considered in this paper are undirected,  simple,  and finite. Let $G$ be a graph.
For two disjoint subsets $S$ and $T$ of $V(G)$,  $e_{G}(S,T)$ denotes the number of edges in $G$ with one endvertex in $S$ and the other in $T$. If $S=\{x\}$, we simply write $e_{G}(\{x\},T)$ as $e_{G}(x,T)$. Also
for a subgraph $D$ of $G$ with $V(D)\cap S=\emptyset$, we write $e_G(S,V(D))$ as $e_G(S,D)$.
For $S\subseteq V(G)$, we denote by $G[S]$ the subgraph of $G$ induced by $S$ and $G-S$ the subgraph $G[V(G)\setminus S]$. We also let $N_G(S)=\bigcup_{v\in S}N_G(v)$, where $N_{G}(v)$ is the set of vertices that are adjacent to  $v$ in $G$.  For integers $p$
and $q$, we let $[p,q]=\{i\in \mathbb{Z}:  p\le i\le q\}$.

A hypergraph $H$ consists of a vertex set $V(H)$ and an edge set $E(H)$,
where each edge in $E(H)$ is a subset of $V(H)$. We allow loops (size one elements in $E(H)$) and multiple edges (repeated elements in $E(H)$) in $H$ for generality, although they play no role when considering connectivity-type parameters.
For an integer $r\ge 1$,
we say  that $H$ is $r$-\emph{uniform} if every edge of $H$ contains
exactly $r$ vertices.   Thus a  graph is a $2$-uniform hypergraph.
For notational simplicity, for a graph $G$, we denote  an edge $e=\{u,v\}$ by $uv$.
A vertex $v\in V(H)$ is \emph{incident to} an edge $e\in E(H)$
if $v\in e$.   For $S\subseteq V(H)$, $H-S$
is obtained from $H$ by deleting all vertices of $S$ and all edges that are incident to a vertex of $S$.
We define the  incidence bipartite graph $\mathcal{I}(H)$ of $H$  as the bipartite graph with bipartition $(E(H), V(H))$, where an edge $e \in E(H)$ is adjacent to a vertex $v \in V(H)$ in $\mathcal{I}(H)$ if and only if $e$ is incident to $v$ in $H$.
This representation  allows all concepts and properties of $H$ to be reformulated in terms of those of $\mathcal{I}(H)$.
Conversely, given any bipartite graph $G$ with bipartition $X$ and $Y$ such that every $x\in X$ satisfying $d_G(x) \ge 1$,
we can construct a hypergraph $H$ on $Y$  with $E(H)=\{N_G(x): x\in X\}$.  We call $H$
the  \emph{hypergraph associated  with  $G$}.  (Note that even if $G$ is a simple graph, $H$ may contain loops or multiple edges. This is another reason that we allow hypergraphs in this paper to have loops or multiple edges.)

In this paper, we investigate the notions of graph toughness and $k$-factors in bipartite graphs through this incidence perspective.
To this end, we first introduce these  and related concepts for hypergraphs. Let $H$ be a hypergraph.
The notion of   \emph{Berge-graph}  was introduced by Gerbner and Palmer~\cite{GP2017}. Specifically, for a graph $G$,  a hypergraph  $H$   contains a \emph{Berge-$G$} if $V(G) \subseteq V(H)$ and
there exists an injection $\varphi: E(G) \rightarrow E(H)$ such that $e\subseteq \varphi(e)$ for every $e\in E(G)$. In other words, for a graph $G$ and a hypergraph $H$ with $V(G)\subseteq V(H)$, we say
that a subhypergraph $H_1$ of $H$ is a \emph{Berge}-$G$ if each hyperedge $S$ in $H_1$ can be mapped to an edge whose
endvertices are both contained in $S$, such that the resulting graph is $G$.
In particular,  for a given integer $k\ge 1$, if $H$ contains a   Berge-$k$-regular graph $G$ with $V(G)=V(H)$, then we say that $H$
has a \emph{Berge-$k$-factor}.  We  say $H$ is a \emph{complete hypergraph} if it contains a  Berge-complete graph $G$ with $V(G)=V(H)$.  Two vertices  $u,v$ of $H$
are \emph{connected} if  there is a  Berge-path connecting $u$ and $v$.
It is routine to check that this binary  relation is an equivalence relation on $V(H)$.
Thus, it partitions   $V(H)$ into subsets of equivalence classes.
Each subhypergraph of $H$ on the same equivalence class is a \emph{component} of $H$.
Denote by $c(H)$ the number of
components of $H$.
The hypergraph $H$  is \emph{connected} if $c(H)=1$.

Let $t\ge 0$ be a real number.  The  hypergraph $H$ is  \emph{$t$-tough} if for every subset $S\subseteq V(H)$ with $c(H-S)\ge 2$, the inequality $|S|\ge t\cdot
c(H-S)$ holds. The \emph {toughness $\tau(H)$} is defined as  the largest real number $t$ for which $H$ is
$t$-tough, or $\tau(H)=\infty$ if $H$ is complete. Introduced by  Chv\'{a}tal in 1973~\cite{C1973}, graph toughness has served as a cornerstone for establishing sufficient conditions for Hamiltonian cycles and  related substructures.

Now we  are ready to define notions of toughness and Berge-$k$-factor in  bipartite graphs through the incidence perspective.
Let $G[X,Y]$ be a bipartite graph with bipartition $(X,Y)$, where  $G$ has no isolated vertex  in $X$.  We will
assume throughout this paper  that elements of $Y$ correspond to the vertices of the  hypergraph  associated with $G$.  For $S\subseteq Y$, we define $G\ominus S=G-\left(S\cup (\bigcup_{v\in S}  N_G(v))\right)$, and call this a \emph{strong-deletion} of $S$.  This operation corresponds to deleting the set of vertices in $S$  and all
edges incident to them in the hypergraph  associated with $G$.
Notice that every component of $G\ominus S$ contains a vertex of $Y$ when $S \neq Y$ (so the component corresponds to a component of the hypergraph associated with $G$), as $G$
has no isolated vertices in $X$.  If $c(G\ominus S) \ge 2$, we call $S$ a \emph{$Y$-cutset} of $G$.
The \emph{$Y$-toughness} of $G$, denoted by $\tau_{Y}(G)$, is the maximum real number $t$ such that $|S|\geq t\cdot c(G\ominus S)$ for any $Y$-cutset $S$ if $S$ exists or is $\infty$ if $G$ has no $Y$-cutset. By construction,
the $Y$-toughness of $G$ is exactly the toughness of the hypergraph associated with $G$.
Let $k\geq 1$ be an integer. A $(\{0,2\},k)$-\emph{factor} of $G$ is a  subgraph $F\subseteq G$ such that for any $x\in X$, $d_{F}(x)\in \{0,2\}$ and for any $y\in Y$, $d_{F}(y)=k$.
From the definition, we see that  a Berge-$k$-factor of a hypergraph $H$ is a $(\{0,2\},k)$-factor of $\mathcal{I}(H)$ and vice versa.
Our main result is the following.

\begin{thm}\label{thm2}
	Let $G[X,Y]$ be a  bipartite graph with  no isolated vertices in $X$ and $k\geq 1$ be an integer. If $\tau_{Y}(G)\geq k$,  $k|Y|$ is even
	and  $|Y|\geq k+1$,  then $G$ has a  $(\{0,2\},k)$-factor.
\end{thm}

Theorem~\ref{thm2} gives a sufficient  toughness condition for  the existence of Berge-$k$-factors in hypergraphs.

\begin{coro}\label{thm2:k-factor-in-hypergraph}
	For any integer $k\ge 1$,  every $k$-tough hypergraph $H$ has a Berge-$k$-factor if $k|V(H)|$ is even and $|V(H)|\geq k+1$.
\end{coro}

Corollary~\ref{thm2:k-factor-in-hypergraph} implies the celebrated theorem by Enomoto, Jackson, Katerinis, and Saito
from 1985~\cite{EJKS1985}:  For any integer $k\ge 1$,  a $k$-tough graph $G$ has a $k$-factor if $k|V(G)|$ is even  and $|V(G)|\geq k+1$.
The toughness condition in Corollary~\ref{thm2:k-factor-in-hypergraph} is  sharp when $G$ is
a graph, as proved in~\cite{EJKS1985}  by Enomoto, Jackson, Katerinis, and Saito.

The proof of this extension presents a unique challenge due to the inherent non-uniformity of hypergraphs compared to graphs. To overcome this, we first adapt traditional graph-theoretic tools for finding factors to the context of  bipartite graphs. Subsequently, we develop a novel approach for identifying cutsets within hypergraphs. The details of the modified factor results are presented in Section 2, while Section 3 gives  the proof of Theorem~\ref{thm2}.

 \section{Notation and parity factors in bipartite graphs}

To prove Theorem~\ref{thm2},  we wish to find a subgraph $F$ of $G$ such that either $d_F(x)=0$  or $d_F(x)=2$
for each $x\in X$, and $d_F(y)=k$ for each $y\in Y$.  Such a construction is usually done by
utilizing the concept of $(g,f)$-factors,  where $g$ and $f$  are integer-valued functions on $V(G)$
and $g(v)$ and $f(v)$ respectively serve as  a lower and upper bound on the degree of each vertex $v$ in $F$.
Since we require $d_F(x)\in\{0,2\}$ (not allowing $d_F(x)=1$),  we will require
$g(x)=0$, $f(x)=2$, and $d_F(x)$ to have the same parity as $g(x)$ and $f(x)$
for all $x\in X$. This leads to the concept of \emph{parity $(g,f)$-factors}.

Let $G$ be a graph. For a given subset $W\subseteq V(G)$, let $g,f:V(G)\rightarrow \mathbb{Z}$ be two functions such that $g(v)\leq f(v)$ for all $v\in V(G)$ and $g(v)\equiv f(v)$ (mod 2) for all $v\in W$. Then a spanning subgraph $F$ of $G$ is called a \emph{partial parity $(g,f)$-factor} with respect to $W$ if
\begin{center}
\begin{enumerate}[(i)]\item $g(v)\leq d_{F}(v)\leq f(v)$ for all $v\in V(G)$ and
\item $d_{F}(v)\equiv g(v)\equiv f(v)$ (mod 2) for all $v\in W$.
\end{enumerate}
\end{center}

Kano and Matsuda~\cite{KM2001} gave a necessary and sufficient condition for a graph to have a partial parity $(g,f)$-factor.

\begin{thm}[\cite{KM2001}]\label{partial} Let $G$ be a graph, $W\subseteq V(G)$, and  $g,f:V(G)\rightarrow \mathbb{Z}$ be two functions satisfying
\[\text{$g(v)\leq f(v)$ for all   $v\in V(G)$  and   $g(v)\equiv f(v)  \pmod{2}$ for all $v\in W$}.\]
Then $G$ has  a partial parity $(g,f)$-factor with respect to $W$ if and  only if for all disjoint subsets $A$ and $B$ of $V(G)$,
\begin{equation}\label{eq1} \delta_{G}(A,B):=\sum\limits_{v\in A} f(v)-\sum\limits_{v\in B}g(v)+\sum\limits_{v\in B} d_{G-A}(v)-h_{W}(A,B)\geq 0,
\end{equation}
where $h_{W}(A,B)$ denotes the number of components $D$ of $G-(A\cup B)$ such that $g(v)=f(v)$ for all  $v\in V(D)\backslash W$ and
\begin{equation}\label{eq2} \sum\limits_{v\in V(D)}f(v)+e_{G}(D,B)\equiv 1 \pmod{2}.
\end{equation}
\end{thm}

Following the notation in Theorem~\ref{partial}, a component $D$ of $G-(A\cup B)$ with  $g(v)=f(v)$ for all  $v\in V(D)\backslash W$ is called an \emph{$f$-odd component} if it satisfies (\ref{eq2}) and is called an \emph{$f$-even component} if $$\sum\limits_{v\in V(D)}f(v)+e_{G}(D,B)\equiv 0 \pmod{2}.$$

Let $G[X,Y]$ be a bipartite graph and $k\ge 1$ be an integer, and let
$f,g:V(G)\rightarrow \mathbb{Z}$ be \begin{equation}\label{eq3}f(v)=\left\{\begin{array}{cc}
		2, & v\in X; \\
		k, & v\in Y;
	\end{array}
	\right. \quad \text{and} \quad
	g(v)=\left\{\begin{array}{cc}
		0, & v\in X; \\
		k, & v\in Y.
	\end{array}
	\right.\end{equation}
Then a  partial parity $(g,f)$-factor  with respect to $X$   is a $(\{0,2\},k)$-factor in  $G$.

Applying
Theorem~\ref{partial},  if a bipartite graph $G[X,Y]$ does not  have a partial parity $(g,f)$-factor
with respect to $X$,
then there exist disjoint sets $A,B\subseteq V(G)$ for which $\delta_{G}(A,B)<0$.
We call $(A,B)$ a \emph{barrier} of $G$.
A \emph{biased barrier} is a barrier $(A,B)$ such that among all the
barriers, we have (i) $\delta_G(A,B)$ is minimum,  (ii) subject to (i) $|B|$ is minimum,  and  (iii)   subject to (i)   and (ii) $|A|$ is maximum.  Let $(A,B)$ be a biased barrier of $G$.
For $Z\subseteq A\cap X$, let
\begin{equation}\label{eq-hz}q(Z)=|N_G(Z)\cap B|+|\{D: \text{$D$ is an $f$-odd component of $G-(A\cup B)$ and $e_G(Z,D) \ge 1$}\}|.
\end{equation}
When $Z=\{z\}$,
we simply write $q(z)$ for $q(\{z\})$.
We study the structure of a biased barrier in Theorem~\ref{lem2.2},
which is of independent interest.
We first show  that for the partial parity $(g,f)$-factor with respect to $X$
defined in~\eqref{eq3}, we have $\delta_{G}(A,B)\equiv 0 \pmod{2}$  for all disjoint subsets $A, B$ of  $V(G)$
provided that $k|Y|$ is even.

%
\begin{lem}\label{lem:delta-A-B-even}
Let $G[X,Y]$ be a bipartite graph and $k\ge 1$ be an integer, and let
 $W=X$ and $f,g$ be defined in~\eqref{eq3}. If $k|Y|$ is even, then
  $\delta_{G}(A,B)\equiv 0 \pmod{2}$ for all disjoint subsets $A,B\subseteq V(G)$.
\end{lem}

\proof Let  $U=V(G)\setminus (A\cup B)$. Since $W=X$, by the definition of $f$ and $g$,
for any component $D$ of $G-(A\cup B)$,
we have  $f(v)=g(v)$ for all $v\in V(D)\setminus W$. Thus
$$h_W(A,B)= \sum\limits_{\text{$D$ is a component of $G-(A\cup B)$}} \left( \left(\sum\limits_{x\in V(D)}f(x)+e_{G}(D,B)\right) \pmod{2}\right ) ,$$
where $\left(\sum\limits_{x\in V(D)}f(x)+e_{G}(D,B)\right) \pmod{2}$ is taken to be either 0 or 1.  Therefore,  $h_W(A,B) \equiv \sum\limits_{x\in U} f(x)+e_G(B,U) \pmod{2}$, and so
\begin{eqnarray*}
	\delta_{G}(A,B)& \equiv &\sum\limits_{x\in A} f(x)-\sum\limits_{y\in B}g(y)+\sum\limits_{y\in B} d_{G-A}(y)-h_{W}(A,B)\\
	&\equiv & \sum\limits_{x\in A} f(x)-\sum\limits_{y\in B}g(y)+2|E(G[B])|+e_G(B,U)-\sum\limits_{x\in U} f(x)-e_G(B,U)\\
	&\equiv & \sum\limits_{x\in A} f(x)-\sum\limits_{y\in B}g(y)-\left(\sum\limits_{x\in V(G)} f(x)-\sum\limits_{x\in A} f(x)-\sum\limits_{y\in B} f(y)\right) \\
	&&(\text{by}\quad 2|E(G[B])| \equiv 0\pmod{2})	  \\
	&\equiv & -\sum\limits_{y\in B}g(y)-\sum\limits_{x\in V(G)}f(x)+\sum\limits_{y\in B} f(y) \\
	&\equiv& -k|B\cap Y|-(2|X|+k|Y|)+(2|B\cap X|+k|B\cap Y|) \\
	&\equiv & 0 \pmod{2},
\end{eqnarray*}
where the equation in the last line follows from the assumption $k|Y|\equiv 0 \pmod{2}$.
\qed

\begin{thm}\label{lem2.2} Let $G[X,Y]$ be a bipartite graph and $k\geq 1$ be an integer. Suppose that  $k|Y|$ is even and $G$
	does not have a partial parity $(g,f)$-factor with respect to $W=X$, where $g$ and
	$f$ are defined in~\eqref{eq3}.
Let $(A,B)$ be  a biased barrier of $G$.  Then the following statements hold:
\begin{enumerate}[\rm(i)]
	\item $B\subseteq Y$;
	\item For every $f$-odd component $D$ of $G-(A\cup B)$ and $v\in V(D)$, $e_{G}(v,B)\leq 1$;
\item For every $f$-even component $D$ of $G-(A\cup B)$ and $v\in V(D)$, $e_{G}(v,B)=0$;
\item For any $Z\subseteq A\cap X$ with $N_G(Z)\cap B=\emptyset$, we have $q(Z) \ge 2|Z|$.
\end{enumerate}
\end{thm}

\begin{proof}
	Let  $U=V(G)\setminus (A\cup B)$.
For (i), suppose that there exists a vertex $v\in B\cap X$.
As  $h_{W}(A,B\setminus\{v\})\geq h_{W}(A,B)-e_{G}(v,U)$, we get
\begin{eqnarray*}
\delta_{G}(A,B\setminus\{v\})&=&\sum\limits_{x\in A} f(x)-\sum\limits_{y\in B\setminus\{v\}}g(y)+\sum\limits_{y\in B\setminus\{v\}} d_{G-A}(y)-h_{W}(A,B\setminus\{v\})\\
&\leq &\sum\limits_{x\in A} f(x)-\sum\limits_{y\in B}g(y)+g(v)+\sum\limits_{y\in B} d_{G-A}(y)-e_{G}(v,B\setminus\{v\})-h_{W}(A,B)\\
&=&\delta_{G}(A,B)+g(v)-e_{G}(v,B\setminus \{v\})\\
& \le  &\delta_{G}(A,B),
\end{eqnarray*}
as $g(v)=0$.  Thus $(A,B\setminus\{v\})$ is a barrier of $G$ with $\delta_{G}(A,B\setminus\{v\}) \le \delta_{G}(A,B)$
and with $|B\setminus \{v\}|<|B|$. This gives
 a contradiction to the choice of $(A,B)$.

We prove statements (ii) and (iii) together. Let $D$ be a component of $G-(A\cup B)$.
As $G$ is bipartite, $e_{G}(v,B)=0$ if $v\in V(D)\cap Y$. So assume that $v\in V(D)\cap X$. By the assumption that $(A,B)$ is a biased barrier and by Lemma~\ref{lem:delta-A-B-even}, we know that $\delta_{G}(A\cup \{v\},B)\geq \delta_{G}(A,B)+2$. Therefore,
\begin{eqnarray*}
 &&\delta_{G}(A,B)+2\leq\delta_{G}(A\cup \{v\},B)\\&=&\sum\limits_{x\in A\cup \{v\}} f(x)-\sum\limits_{y\in B}g(y)+\sum\limits_{y\in B} d_{G-(A\cup \{v\})}(y)-h_{W}(A\cup \{v\},B)\\
&=& \sum\limits_{x\in A} f(x)+f(v)-\sum\limits_{y\in B}g(y)+\sum\limits_{y\in B} d_{G-A}(y)-e_{G}(v,B)-h_{W}(A\cup \{v\},B)\\
&  \le & \begin{cases}
\sum\limits_{x\in A} f(x)+f(v)-\sum\limits_{y\in B}g(y)+\sum\limits_{y\in B} d_{G-A}(y)-e_{G}(v,B)-h_{W}(A,B)+1 &  \text{if $D$ is $f$-odd}; \nonumber \\
\sum\limits_{x\in A} f(x)+f(v)-\sum\limits_{y\in B}g(y)+\sum\limits_{y\in B} d_{G-A}(y)-e_{G}(v,B)-h_{W}(A,B) &  \text{if $D$ is $f$-even}; \nonumber
\end{cases} \\
&=&\begin{cases}
\delta_{G}(A,B)+2-e_{G}(v,B)+1&  \text{if $D$ is $f$-odd};\\
\delta_{G}(A,B)+2-e_{G}(v,B)&  \text{if $D$ is $f$-even}.
\end{cases}
\end{eqnarray*}
We obtain that
  $e_{G}(v,B)\leq 1$ when $D$ is an $f$-odd component  of $G-(A\cup B)$ and
 $e_{G}(v,B)=0$ when $D$ is an $f$-even component of $G-(A\cup B)$.

 To prove (iv), suppose to the contrary that $q(Z) \le  2|Z|-1$. By the assumption that $(A,B)$ is a biased barrier, we know that $\delta_{G}(A\setminus Z,B) \ge \delta_G(A,B)$. As $h_W(A\setminus Z,B) \ge h_W(A,B)-q(Z) $ by the definition of $q(Z)$ in Equation~\eqref{eq-hz} and $e_G(B,Z)=0$ ($N_G(Z)\cap B=\emptyset$), we get
 \begin{eqnarray*}
 \delta_G(A,B) &\le& \delta_{G}(A\setminus Z,B) =\sum\limits_{x\in A} f(x)-2|Z|-\sum\limits_{y\in B}g(y)+\sum\limits_{y\in B} d_{G-A}(y)-h_{W}(A\setminus Z,B)\\
 	& \le & \delta_G(A,B)-2|Z|+q(Z)\\
 	&\le & \delta_G(A,B)-1,
 \end{eqnarray*}
a contradiction.
\end{proof}

\section{Proof of Theorem~\ref{thm2}}

Assume to the contrary that $G$ has no $(\{0,2\},k)$-factor.  Thus for $W=X$
and $f$ and $g$ defined as in~\eqref{eq3},  $G$
has a biased barrier $(A,B)$ by Theorem~\ref{partial}.
Since $W=X$, by the definitions of $g$ and $f$,
a component $D$ of $G-(A\cup B)$ is $f$-odd ($f$-even)
if $k|V(D)\cap Y|+e_G(D,B)$  is odd (even).

If $G$ has no $Y$-cutset, let $S$ be any subset of $Y$ with $|S|=|Y|-2$,  and  let $Y\setminus S=\{u,v\}$. Then $u$ and $v$ belong to the same component in $G\ominus S$. Thus for any two distinct vertices $u,v\in Y$, there is a  vertex $x\in X$ such that $x$ is adjacent in $G$ only to $u$ and $v$.  For any vertex $x\in X$ with $d_G(x)=2$, if we remove $x$ from $G$ and
add a new edge joining the two vertices in $N_G(x)$, then we get a graph $G'$ such that $G'[Y]$ contains a spanning complete subgraph. Thus $G'$ has a $k$-factor and so $G$ contains a $(\{0,2\},k)$-factor.  Therefore we assume that $G$ has a $Y$-cutset. Our goal is
to find a $Y$-cutset that violates the $Y$-toughness of $G$.
Under the assumption that $G$ has a $Y$-cutset, we show in the claim below that the size of $Y$ is relatively large.

\begin{claim}\label{claim:Y-size}
	We have $|Y| \ge 2k+2$.
\end{claim}

\proof Suppose to the contrary that $|Y| \le 2k+1$.
As $G$ has a $Y$-cutset, there exist distinct $y_1, y_2\in Y$ such that
there is no $x\in X$ with $N_G(x)=\{y_1,y_2\}$.  Let $S=Y\setminus\{y_1,y_2\}$.
Then $G\ominus S$ has exactly  two trivial components: one consists of $y_1$
and the other consists of $y_2$. However, $|S|/c(G\ominus S) \le (2k-1)/2 <k$,
contradicting $\tau_Y(G) \ge k$.
\qed

Let $\mathcal{C}$ be the set of all $f$-odd components of $G-(A\cup B)$, and
 $\mathcal{C}_1$ be the set of all $f$-odd components $D$ of $G-(A\cup B)$ such that $|V(D)|=1$ and $V(D)\subseteq X$.
(As a component $D\in \mathcal{C}_1$ does not correspond to a component of the hypergraph associated
 with $G$, we will ``exclude'' those in our analysis.)
 By Theorem~\ref{lem2.2}(ii) and Equation~\eqref{eq2}, $e_{G}(D,B)=1$
 for any $D\in \mathcal{C}_1$. Let $$U=V(G)\setminus (A\cup B) \quad \text{and} \quad U_0=\bigcup_{D\in \mathcal{C}_1} V(D).$$

Let $\mathcal{C}\setminus\mathcal{C}_1=\{D_1,D_2,\ldots,D_{\ell}\}$.
Then we have $h_W(A,B)=\ell+|\mathcal{C}_1|$ as $h_W(A,B)=|\mathcal{C}|$.
For those components $D_i$, we have the claim below.

\begin{claim}\label{claim0} For each $D_i$ with $i\in [1,\ell]$, the following statements hold.
	\begin{enumerate}[{\rm(i)}]
		\item If $|V(D_i)|=1$,  then  $k\equiv 1\pmod 2$.
		\item $|V(D_i)| \ge e_G(B,D_i)$.
		\item Let
		$x\in V(D_i)\cap N_G(B)$.  Then $x$ has a neighbor in $V(D_i)\cap Y$.
	\end{enumerate}
\end{claim}

\begin{proof}
	Since $D_i\not\in \mathcal{C}_1$, we know that  $V(D_i)=\{y\}$ for some $y \in  Y$. Then we have $e_{G}(D_i,B)=0$ as $B\subseteq Y$ by Theorem~\ref{lem2.2}(i).
	From  Equation~\eqref{eq2} with $W=X$, we have $f(y)+e_{G}(D_i,B)\equiv 1\pmod 2$.  As $f(y)=k$, we get $k\equiv 1\pmod 2$.
	This proves (i).
	For (ii),  as $B\subseteq Y$ by Theorem~\ref{lem2.2}(i), we know that $V(D_i)\cap N_G(B) \subseteq V(D_i)\cap X$.  Then every $x\in V(D_i)\cap N_G(B)$ satisfies  $e_G(x, B) =1$ by Theorem~\ref{lem2.2}(ii). Hence $|V(D_i)\cap N_G(B)|= e_G(B,D_i)$ and so $|V(D_i)| \ge e_G(B,D_i)$.
	For (iii), we have $x\in V(D_i)\cap X$ as  $B\subseteq Y$.
	Since $D_i\not\in \mathcal{C}_1$, we get $V(D_i) \ne \{x\}$.
	As $D_i$ is a component of $G-(A\cup B)$ and so is connected,
	we know that $x$ has a neighbor in  $V(D_i)\cap Y$.
\end{proof}

Since $f$ and $g$ are defined differently for vertices of $X$
and vertices of $Y$,  we partition $A$ into two subsets: $A_1=A\cap X$ and $A_2=A\cap Y$.
Denote by
$$  a_1=|A_1|, \quad  a_2=|A_2|, \quad \text{and} \quad b=|B|.$$ Then
\begin{eqnarray}0&> &\delta_{G}(A,B)=\sum\limits_{x\in A} f(x)-\sum\limits_{y\in B}g(y)+\sum\limits_{y\in B}d_{G-A}(y)-h_{W}(A,B) \nonumber \\
&=&2|A\cap X|+k|A\cap Y|-k|B|+e_{G}(B,U)-h_{W}(A,B)\nonumber \\
&=&2a_1+ka_2-kb+e_{G}(B,U)-\ell-|\mathcal{C}_1|\nonumber \\
&=&2a_1+ka_2-kb+e_{G}(B,U\setminus U_0)-\ell. \quad  (\text{By $e_{G}(D,B)=1$
for any $D\in \mathcal{C}_1$}.) \label{eq6}
\end{eqnarray}

For a subset $X^*$ of $X$, if we need to ``cut off" edges leaving vertices of $X^*$, we need to strong-delete some vertices in $N_G(X^*)$. By taking a vertex in $N_G(x)$ for each vertex $x$ of $X^*$,  we can find a subset $Y^*$ of $N_G(X^*)$
with $|Y^*| \le |X^*|$ such that $X^*\subseteq N_G(Y^*)$. We call $Y^*$ a \emph{$Y$-dominating set} of $X^*$.
Note that a minimal $Y$-dominating set of $X^*$ has size at most $|X^*|$.  In our proof, we will strong-delete $Y$-dominating sets of
subsets of $A$  in order to remove the connections of two vertices of $B$ through vertices of $A$.
 We consider two cases regarding the value of $k$.

 \smallskip
 {\bf \noindent Case 1}: $k=1$.
 \smallskip

Then  from~\eqref{eq6},  we have
\begin{equation}
b-a_1>a_1+a_2+e_G(B,U\setminus U_0)-\ell. \label{eqn:b-a1>a1+a2-}
\end{equation}
Our goal is to find a strong-deletion with at most  $a_1+a_2+e_G(B,U\setminus U_0)-\ell$
vertices of $Y$ which creates at least $b-a_1$ components to achieve a contradiction to
$\tau_Y(G) \ge 1$.

For each vertex $x\in A_1$  such that $x$ is adjacent in $G$ to a vertex of $B$, we let
$v_x\in B$ be an arbitrary vertex in $N_G(x)\cap B$.  For each $x\in A_1$  such that $x$ is adjacent in $G$ only  to a vertex in $U$,
we let
$v_x\in U$ be an arbitrary vertex in $N_G(x)\cap U$.
Let $Z_1=\{v_x:  \text{$x\in A_1$}\}$. Then $|Z_1|\le |A_1|$ and
 $N_G(Z_1)\cap A_1=N_G(B \cup U)\cap A_1$.
 For each  component $D_i\in \mathcal{C}\setminus \mathcal{C}_1$,  we will strong-delete some $Y$-vertices
 of $D_i$ to cut off its connection to all but at most one vertex of $B$.
If  $e_G(B,D_i) \ge 2$, then by Theorem~\ref{lem2.2}(ii), we can take a subset $W^*_i\subseteq N_G(B)\cap V(D_i)$ of $e_G(B,D_i)-1$  vertices. Let $W_i\subseteq N_{D_i}(W^*_i)$ be minimal  $Y$-dominating set of $W_i^*$.  Then  we have    $W_i\subseteq Y$ and $|W_i| \le |W^*_i| =e_G(B,D_i)-1$.
 If $e_{G}(B,D_i)\leq 1$, we simply let
 $W_i=\emptyset$.

 It is possible that some component $D_i$ for $i\in [1,\ell]$ satisfies $e_G(B, D_i)=0$.
 Thus,  in order to compare $e_G(B,U\setminus U_0)$ and $\ell$ later on, we need to distinguish
 such components.
 Let $\ell^*$ be the number of components in $\mathcal{C}\setminus \mathcal{C}_1$ whose vertices are  not adjacent in $G$ to any vertex of $B$, and let $U^*$ be the set of vertices of those components.  Let  $Z_1^*=Z_1\cap U^*$. Then  $|Z_1\cap B| \le a_1-|Z_1^*|$.
 Let
$$
G_1=G\ominus\left(A_2\cup Z_1\cup\left(\bigcup\limits_{i=1}^{\ell}W_i\right)\right).
$$
All the $Y$-vertices not contained in $A_2\cup Z_1\cup\left(\bigcup\limits_{i=1}^{\ell}W_i\right)$ remain
in $G_1$. In particular, all vertices of $B\setminus Z_1$ remain in $G_1$.
All the $X$-vertices which are not neighbors of vertices in $A_2\cup Z_1\cup\left(\bigcup\limits_{i=1}^{\ell}W_i\right)$
in $G$ are also remained in $G_1$.
Since $e_{G\ominus W_i}(B, D_i\ominus W_i) \le 1$, it follows that
$$c(G_1) \ge b-|Z_1\cap B|+(\ell^*-|Z_1^*|) \ge b-(a_1-|Z_1^*|)+\ell^*-|Z_1^*|
=b-a_1+\ell^*.$$
On the other hand, we have
\begin{eqnarray*}
&& \left| A_2\cup Z_1\cup\left(\bigcup\limits_{i=1}^{\ell}W_i\right)\right|\\
& \le& a_2+|Z_1|+e_G(B,U\setminus U_0)-\ell+\ell^*\\
& <&-2a_1+b+|Z_1|+\ell^* \quad \text{(by~\eqref{eqn:b-a1>a1+a2-})} \\
&\le& b-a_1+\ell^*.
\end{eqnarray*}

If $b-a_1+\ell^*\ge 2$, then we get a contradiction to $\tau_Y(G)\ge 1$.
 Thus we assume that
$b-a_1 +\ell^* \le 1$.   Together with~\eqref{eqn:b-a1>a1+a2-}, we have $1\geq b-a_1+\ell^* >a_1+a_2+e_G(B,U\setminus U_0)-\ell +\ell^* \ge 0$.  This gives   $a_1=a_2=0$, $e_G(B,U\setminus U_0)+\ell^*=\ell$ and $b+\ell^*=1$.
If $b=0$,  then we get $(A,B)=(\emptyset,\emptyset)$. Thus $G$ is an $f$-odd component as $\delta_{G}(A,B)=-h_{W}(\emptyset,\emptyset)<0$. This is a contradiction to Equation~\eqref{eq2} and the assumption that $k|Y|$ is even.
 Thus $|B|=1$ and $\ell^*=0$.
 As $A=A_1\cup A_2=\emptyset$, we have $e_G(D,B) \ge 1$
 for any component $D$ of $G-(A\cup B)$. Thus by Theorem~\ref{lem2.2}(ii),
 we have $c(G-(A\cup B)) = \ell +|\mathcal{C}_1|$.
 If $\ell\ge 2$,
 then $G\ominus B$ has at least two components
 and so $\tau_Y(G) \le \frac{1}{2}$, a contradiction.  Thus $\ell=1$
 as   $|Y| \ge 4$ by Claim~\ref{claim:Y-size} and the fact that $|B\cup A_2|=1$.
 Since $e_G(B,U\setminus U_0)-\ell +\ell^*=0$,   $\ell^*=0$,  and $\ell = 1$,
 it follows that $e_G(B,U\setminus U_0)=1$, and $e_G(B,U\setminus U_0)=e_G(B,V(D_1))$.
Let $y\in V(D_1)\cap Y$ be a vertex such that the vertex in $N_G(B)\cap V(D_1)$
is adjacent in $D_1$ to $y$.  Then $G\ominus\{y\}$ has at least two components: one is the vertex in $B$,
and the rest contains vertices in $D\ominus\{y\}$.
This again gives  a contradiction
to $\tau_Y(G) \ge 1$. Therefore, when $k=1$, $G$ has a $(\{0,2\},1)$-factor.

\smallskip

{\bf \noindent Case 2}: $k\ge 2$.
\smallskip


The  proof of this case basically  follows the framework used in the proof of the main theorem in~\cite{EJKS1985} by   Enomoto,  Jackson, Katerinis, and Saito. However,  the proof is more involved as
vertices of $X$ can have various degrees in $G$, which in contrast, can only have degree 2 under the setting in~\cite{EJKS1985}.   We will start by strong-deleting vertices of $B$ which have at least $k$
neighbors in $A_1$ iteratively. After this process, the remaining vertices of $B$ will have at most $k-1$
neighbors in the reduced set of $A_1$. Then we will take maximal subsets of $B$, where no two vertices
of the set are adjacent to a common vertex of $A_1$, iteratively. Vertices of each of these sets will play the role of components after we appropriately strong-delete some other $Y$-vertices.
Each time we use  the toughness condition to  provide
an upper bound on the size of the  set.  At the end, a contradiction is
 achieved by comparing an inequality obtained  by combining all the upper bounds
and Inequality~\eqref{eq6}.

 In the first step, we define $B_1 \subseteq B$ iteratively as follows:
For a vertex $y\in B$, if $|N_G(y)\cap A_1| \ge k$,
we place $y$ in $B_1$. Then we continue this process by updating $G$ as $G\ominus B_1$, $B$ as $B\setminus B_1$, and $A_1$ as $A_1\setminus N_G(B_1)$. This process will stop as $G$ is a finite graph. When it stops,
we let
\begin{eqnarray*}
	G_1=G\ominus B_1, \quad
	B_2=B\setminus B_1,      \quad A_{11}=A_{1}\cap N_G(B_1), \quad A_{12}=A_{1}\setminus A_{11}.
\end{eqnarray*}
By the definition of $B_1$ and $G_1$, we know that
\begin{eqnarray}
	|N_{G_1}(y)\cap A_{12}| &\le &k-1 \quad  \text{for every $y\in B_2$}, \label{eqn-R12} \\
	|A_{11}| &\ge & k|B_1|. \label{eqn-R1b2}
\end{eqnarray}

Let $A_{12}^{1}=\{x\in A_{12}: |N_G(x)\cap B_2|\geq 2\}$. Then define $S_1$ to be a maximal subset of $B_2$ such that $N_G(x)\cap B_2 \not\subseteq S_1$ for every $x \in A_{12}^{1}$. Then for every $y'\in B_2\setminus S_1$, there exists $x'\in A^1_{12}$ such that $y'\in N_G(x')$ and $N_G(x')\subseteq S_1\cup \{y'\}$. Let $T_1 = B_2 \setminus S_1$ and
\[A_{12}^{2}= \{x \in A_{12}^{1}: N_G(x)\cap S_1 = \emptyset\}= A^1_{12}\setminus N_G(S_1).\]
Then define $S_2$ to be a maximal subset of $T_1$ such that $N_G(x)\cap T_1 \not\subseteq S_2$ for every $x\in A_{12}^{2}$. Inductively, for an integer $i \geq 2$, let
\begin{eqnarray*}
	T_i &= & T_{i-1}\setminus S_i=T_1\setminus (S_1\cup \ldots \cup S_i), \\
	A_{12}^{i+1}&=& \{x \in A_{12}^{i}: N_G(x)\cap (S_1\cup\cdots\cup S_i)=\emptyset\} =A^1_{12}\setminus N_G(S_1\cup \ldots \cup S_i),
\end{eqnarray*}
 and define $S_{i+1}$ to be a maximal subset of $T_i$ such that $N_G(x)\cap T_i \not\subseteq S_{i+1}$ for every $x\in A_{12}^{i+1}$. Finally, let  $S_k=T_{k-1}$ and let $T_k=\emptyset$.

By the definition of $S_i$ and $T_i$, we see that $\{S_1, S_2, \ldots, S_{k-1}, S_{k}\}$
is a partition of $B_2$.    Furthermore,  for each $i\in [1,k-1]$, by the maximality of $S_i$,   for any $v\in T_i$, there  exists $x_i\in A^{i+1}_{12}$
such that $v\in N_G(x_i)$ and $N_G(x_i)\cap B_2\subseteq S_i\cup \{v\}$.
For each $i\in [1,k]$,  $S_i$ may contain vertices $y$ such that there exists $x\in A_{12}\setminus A_{12}^1$
with  $N_G(x)\cap B_2=\{y\}$. To distinguish such vertices $y$, we  create a finer partition of $S_i$.
Let
$$
S_i^1=\{y\in S_i:  \text{there exists $x\in N_G(y)\cap (A_{12}\setminus A_{12}^1)$ with $N_G(x)\cap B_2=\{y\}$}\}, \quad S_i^2 =S_i\setminus S_i^1.
$$
Corresponding to $S_i^1$, we
let  $C_i=\{x\in A_{12}\setminus A^1_{12}: |N_G(x)\cap B_2|=|N_G(x)\cap S_i^1| = 1\}$.
By the definition, it is clear that $|C_i| \ge |S_i^1|$.
Finally, let
$$
Z=\{x\in A_{12}:  N_G(x)\cap B_2=\emptyset\}, \quad \text{and} \quad A_1^*=A_{12}\setminus (Z\cup \bigcup_{i=1}^k C_i).
$$
Note that $A_1^*=A^1_{12}$.
See Figure~\ref{fig2} for a depiction of these sets defined above.
We first show that $S_k$ has the same property as other $S_i$'s.
\begin{claim}\label{claim:Sk-ind}
	It holds that $N_G(S_k)\cap A^k_{12}=\emptyset$, namely, $N_G(A^k_{12})\cap S_k=\emptyset$.
\end{claim}

\proof  By the maximality of $S_i$, $i\in [1,k-1]$, for each $v\in S_k$,
we know that  there exists $x_i\in A^i_{12}$  such that $N_{G}(x_i)\cap B_2 \subseteq S_i \cup \{v\} $.  As $|N_{G}(x_i)\cap B_2| \ge 2$ and $S_i\cap S_j=\emptyset$ for distinct $i, j\in[1,k-1]$,  it follows
that $x_1, \ldots, x_{k-1}$ are all distinct. By~\eqref{eqn-R12} that 	$|N_{G_1}(v)\cap A_{12}| \le k-1$,
we know that  $N_{G_1}(v)=\{x_1, \ldots, x_{k-1}\}$. Thus the claim holds.
\qed

	\begin{figure}[!htb]
	\begin{center}
			\begin{tikzpicture}
				\begin{scope}[shift={(0,0)}]
			\tikzstyle{every node}=[font=\LARGE]
			\draw  (3.75,11.75) rectangle (15.25,10.5);
			\draw (5,11.75) to[short] (5,10.5);
			\draw [](6.25,11.75) to[short] (6.25,10.5);
			\draw [](7.5,11.75) to[short] (7.5,10.5);
			\draw [](10,11.75) to[short] (10,10.5);
			\draw [](11.25,11.75) to[short] (11.25,10.5);
			\node [font=\large] at (4.5,11) {$A_{11}$};
			\node [font=\large] at (5.5,11) {$C_1$};
			\node [font=\large] at (7,11) {$C_2$};
			\node [font=\large] at (4.75,8.5) {$B_1$};
			\draw [short] (5,7.75) -- (5,5.75);
			\draw [short] (12.5,6.75) -- (12.5,5.75);
			\draw [short] (5,13) -- (5,12);
			\draw [short] (15.25,13) -- (15.25,12);
			\draw [->, >=stealth] (5,12.5) -- (9.5,12.5);
			\draw [->, >=stealth] (15.25,12.5) -- (10.5,12.5);
			\draw [->, >=stealth] (5,6.25) -- (8.5,6.25);
			\draw [->, >=stealth] (12.5,6.25) -- (9.25,6.25);
			\node [font=\Large] at (8.8,6.25) {$B_2$};
			\draw [short] (5,8.5) -- (5,8.5);
			\draw  (4.25,9.25) rectangle (12.5,8);
			\draw [short] (5,9.25) -- (5,8);
			\draw [short] (6.25,9.25) -- (6.25,8);
			\draw [short] (7.5,9.25) -- (7.5,8);
			\draw [short] (10,9.25) -- (10,8);
			\node [font=\large] at (10.5,11) {$C_k$};
			\draw [short] (12.25,11.75) -- (12.25,10.5);
			\node [font=\large] at (11.75,11) {$Z$};
			\node [font=\large] at (13.75,11) {$A_1^*=A^1_{12}$};
				\node [font=\Large] at (15.75,11) {$A_1$};
			\node [font=\normalsize] at (5.25,8.75) {$S_1^1$};
			\node [font=\normalsize] at (5.75,8.25) {$S_1^2$};
			\draw [short] (6.25,9.25) -- (5,8);
			\draw [short] (7.5,9.25) -- (6.25,8);
			\draw [short] (10,8) -- (12.5,9.25);
			\node [font=\normalsize] at (6.5,8.75) {$S_2^1$};
			\node [font=\normalsize] at (7,8.25) {$S_2^2$};
			\node [font=\normalsize] at (10.5,8.75) {$S_k^1$};
			\node [font=\normalsize] at (12,8.5) {$S_k^2$};
				\node [font=\Large] at (13,8.5) {$B$};
			\node [font=\Large] at (10,12.5) {$A_{12}$};
			\draw [short] (6.25,7.7) -- (6.25,6.25);
			\draw [short] (12.5,7.25) -- (12.5,6.75);
			\draw [->, >=stealth] (6.25,7) -- (8.75,7);
			\draw [->, >=stealth] (12.5,7) -- (9.75,7);
			\node [font=\Large] at (9.25,7) {$T_1$};
			\draw [short] (7.5,7.75) -- (7.5,7);
			\draw [short] (12.5,7.75) -- (12.5,7.25);
			\draw [->, >=stealth] (7.5,7.5) -- (9.75,7.5);
			\draw [->, >=stealth] (12.5,7.5) -- (10.75,7.5);
			\node [font=\Large] at (10.25,7.5) {$T_2$};
			\draw [short] (4.5,10.5) -- (4.5,9.25);
			\draw [short] (4.75,10.5) -- (4.75,9.25);
			\draw [short] (5.5,10.5) -- (5.25,9.25);
			\draw [short] (5.75,10.5) -- (5.5,9.25);
			\draw [short] (7,10.5) -- (6.5,9.25);
			\draw [short] (7.25,10.5) -- (6.75,9.25);
			\draw [short] (10.5,10.5) -- (10.25,9.25);
			\draw [short] (10.75,10.5) -- (10.5,9.25);
			\draw [line width=1.3pt, dashed] (8.5,11) -- (9,11);
			\draw [line width=1.3pt, dashed] (8.5,8.5) -- (9,8.5);
		\end{scope}
\end{tikzpicture}
\end{center}
\caption{Subsets of $A_1$ and $B$ created in Case 2}
\label{fig2}
\end{figure}

In the following two claims, we establish an upper bound on the size of  $S_i$ for each $i\in[1,k]$.

\begin{claim}\label{claim1} $k|S_1|\leq  |B_1|+a_2+ |T_1|+2|C_1|+2|Z|+e_{G}(U\setminus U_0,S_1)-\ell.$
\end{claim}
\begin{proof}
By Theorem~\ref{lem2.2}(iii), in $G$, all the edges between $S_1$ and $U\setminus U_0$ are exactly the
edges between $S_1$ and vertices of components in $\mathcal{C} \setminus \mathcal{C}_1$. Without loss of generality,  we let $D_1, \ldots, D_{\ell_1}$ be all the  components in $\mathcal{C} \setminus \mathcal{C}_1$
such that $e_{G}(S_1, D_i)\ge 1$ for each $i\in [1,\ell_1]$, where $0 \le \ell_1 \le \ell$.
Assume further that $D_1, \ldots, D_{\ell_2}$ are  all the components among $D_1, \ldots, D_{\ell_1}$
such that $e_{G}(S_1, D_i)\ge 2$ for each $i\in [1,\ell_2]$, where $0\le \ell_2\le \ell_1$.
For each component $D_i$ with $i\in [1,\ell_2]$,  by Theorem~\ref{lem2.2}(ii), we can take a subset $W_i^* \subseteq N_{G}(S_1)\cap V(D_i)$ of $e_{G}(S_1,D_i)-1$ vertices.
Let  $W_i\subseteq N_{D_i}(W_i^*)$ be a minimal $Y$-dominating set of $W_i^*$ (this is possible by Claim~\ref{claim0}). Then $|W_i| \le |W_i^*| =e_{G}(S_1,D_i)-1$.
Let  $U_1^*\subseteq U\cap Y$  be a minimal $Y$-dominating set of $C_1\cup Z$. Then
$|U_1^*| \le |C_1|+|Z|$.
Now we let
$$
G_2={G}\ominus\left(B_1\cup A_2\cup  T_1 \cup U^*_1  \cup \left(\bigcup\limits_{i=1}^{\ell_2}W_i\right)\right).
$$
Then,
for each component $D_i$ with $i\in [1,\ell_1]$,  we have $e_{G_2}(D_i\ominus W_i,V(G_2)\setminus V(D_i))=e_{G_2}(D_i\ominus W_i, S_1)\le 1$. Since every vertex of $N_G(S_1)\cap A_1$
has in $G$ a neighbor from $T_1$ by the definition of $S_1$, it follows that
 all vertices of $N_G(S_1)\cap A_1$ are vanished in $G_2$.
Thus  $c(G_2) \ge |S_1| +\max\{\ell-\ell_1-|U_1^*|, 0\}$.

First consider the case $c(G_2)\geq 2$.  By the toughness of $G$, we get
\begin{eqnarray*}
	k(|S_1|+\max\{\ell-\ell_1-|U_1^*|, 0\}) &\le& |B_1|+\left |A_2\cup T_1\cup U^*_1\cup  \left(\bigcup\limits_{i=1}^{\ell_2}W_i\right) \right| \\
	&\le & |B_1|+a_2+|T_1|+|U^*_1|+e_G(U\setminus U_0, S_1)-\ell_1.
\end{eqnarray*}
As $|U_1^*| \le |C_1|+|Z|$, we then get
\begin{eqnarray*}
	k|S_1| &\le&  |B_1|+a_2+|T_1|+|U^*_1|+e_G(U\setminus U_0, S_1)-\ell_1-k(\ell-\ell_1-|U_1^*|)\\
	&\le &  |B_1|+a_2+|T_1|+|U^*_1|+e_G(U\setminus U_0, S_1)-\ell_1-(\ell-\ell_1-|U_1^*|) \\
	&=& |B_1|+a_2+|T_1|+2|U^*_1|+e_G(U\setminus U_0, S_1)-\ell \\
	&\le & |B_1|+a_2+ |T_1|+2|C_1|+2|Z|+e_{G}(U\setminus U_0,S_1)-\ell,
\end{eqnarray*}
as desired.

Consider now that    $|S_1|+\max\{\ell-\ell_1-|U_1^*|, 0\} \le c(G_2) \le 1$. This gives $|S_1| \le 1$.
By~\eqref{eq6} and~\eqref{eqn-R1b2}, we have $B_2 \ne \emptyset$.
Now by the maximality of $S_1$, we know that $S_1\ne \emptyset$ and
thus $|S_1|=1$.   Therefore,  from $|S_1|+\max\{\ell-\ell_1-|U_1^*|, 0\}   \le 1$, we deduce that $\ell-\ell_1 \le |U_1^*|$.
We also assume to the contrary that $k|S_1|> |B_1|+a_2+ |T_1|+2|C_1|+2|Z|+e_{G}(U\setminus U_0,S_1)-\ell.$
Since
\begin{eqnarray*}
	&&|B_1|+\left |A_2\cup T_1\cup U^*_1\cup  \left(\bigcup\limits_{i=1}^{\ell_2}W_i\right) \right| \\
	&\le & |B_1|+a_2+|T_1|+|U^*_1|+e_G(U\setminus U_0, S_1)-\ell_1 \\
	&= &  |B_1|+a_2+|T_1|+|U^*_1|+(\ell-\ell_1)+e_G(U\setminus U_0, S_1)-\ell \\
	&\le & |B_1|+a_2+|T_1|+2|U^*_1|+e_G(U\setminus U_0, S_1)-\ell \\
	&\le & |B_1|+a_2+ |T_1|+2|C_1|+2|Z|+e_{G}(U\setminus U_0,S_1)-\ell,
\end{eqnarray*}
it follows that
\begin{eqnarray}\label{eqn-cut-size}
|B_1|+\left |A_2\cup T_1\cup U^*_1\cup  \left(\bigcup\limits_{i=1}^{\ell_2}W_i\right) \right| <k|S_1|=k.
\end{eqnarray}

Let $U_1\subseteq U_1^*$  be a minimal $Y$-dominating set of $Z$. Then
$|U_1| \le |Z|$  and  $|U_1| \le |U_1^*|$.
Let $W=A_2\cup B\cup U_1$.  Then we have $|W| \le k$ from~\eqref{eqn-cut-size}, as $|S_1|=1$ and $S_1\cup T_1=B_2$.
By Theorem~\ref{lem2.2}(iv),
we have $\ell \ge 2|Z|$. If $ \ell-|Z|\ge 2$, then
as $|W|\le k$ and $c(G\ominus W) \ge  \ell-|Z|$,
the $Y$-toughness of $G$ implies
\begin{eqnarray*}
	k\ge |W|\ge kc(G\ominus W) \ge k(\ell-|Z| ).
\end{eqnarray*}
This implies  $ \ell-|Z| \le 1$, a contradiction. Thus  $ \ell-|Z|\le 1$.
Since $\ell-|Z| \ge |Z|$, we then have $|Z|\le 1$ and $\ell\le 2$.

Recall that for each component $D_i$ with $i\in [1,\ell_1]$,  we have $e_{G_2}(D_i\ominus W_i,V(G_2)\setminus V(D_i))=e_{G_2}(D_i\ominus W_i, S_1)\le 1$.
We let $W^*$ be a subset of at most $\ell_1$ vertices of $(\bigcup_{i=1}^{\ell_1} V(D_i\ominus W_i))\cap Y$
 that is a $Y$-dominating set of  $N_{G_2}(S_1)\cap (\bigcup_{i=1}^{\ell_1} V(D_i\ominus W_i))$.
Let $W=A_2\cup  T_1 \cup U^*_1  \cup \left(\bigcup\limits_{i=1}^{\ell_2}W_i\right) \cup W^* $.
As $|W^*| \le \ell_1\le \ell \le 2$, from~\eqref{eqn-cut-size}, we know that $|W| \le k+1$.
Then  as $|Y| \ge 2k+2$ by Claim~\ref{claim:Y-size}, $|S_1|=1$, and the vertex of $S_1$
is a component of $G\ominus W$, it follows that $c(G_2) \ge 2$.
However, $|W|/c(G_2) \le (k+1)/2<k$, a contradiction to $\tau_Y(G)\ge k$.
\end{proof}

For each $i\in [2,k]$, we let  $A^*_{1i}=\{x\in A_1^*: N_{G_1}(x)\cap B_2\subseteq \bigcup_{j=1}^{i}S_j, N_{G_1}(x)\cap( \bigcup_{j=1}^{i-1}S_j) \ne \emptyset, N_{G_1}(x)\cap S_i \ne \emptyset\}$. By the definition,
we have $A^*_{1i} \cap A_{1j}^*=\emptyset $ for distinct $i,j\in [2,k]$. Therefore,
\begin{equation}\label{eqn:sum-of-A_{1i}}
	\sum_{i=2}^k|A_{1i}^*| \le |A_1^*|.
\end{equation}

\begin{claim}\label{claim2} For  each $i\in [2, k]$, we have  $$k|S_{i}|\leq  |B_1|+a_2+|C_i|+|T_i|+|A^*_{1i}|+e_{G}(S_i,U\setminus U_0).$$
\end{claim}
\begin{proof} If $S_i=\emptyset$, then the claim holds vacuously. So we assume  $S_i\neq \emptyset$.
Let $W^*_i \subseteq \bigcup_{j=1}^{i-1}S_j$  be a minimal $Y$-dominating set of $A^*_{1i}$.
Then $|W_i^*| \le |A^*_{1i}|$.
Furthermore,  we let   $U_i^*\subseteq U\cap Y$   be a minimal $Y$-dominating set of $C_i$.  Then
 $|U_i^*| \le |C_i|$.
We assume, without loss of generality, that for some $\ell_1\in [0, \ell]$,
$D_1, \ldots, D_{\ell_1}$ are all the components among $D_1, \ldots, D_\ell$ that each is connected to some vertex of $S_i$
by an edge of $G_1$.
For each $j\in [1,\ell_1]$,
 we let $W_j\subseteq V(D_j)\cap Y$ be a minimal $Y$-dominating set of $N_{G_1}(S_i)\cap V(D_j)$.
 Then $|W_j| \le e_{G_1}(S_i, D_j)$ by Theorem~\ref{lem2.2}(ii).
 Let $W=B_1\cup A_2\cup U^*_i\cup T_i\cup W^*_i\cup (\bigcup_{i=1}^{\ell_1}W_i)$
and $G_2=G\ominus W$.
Then we have $c(G_2) \ge |S_i|$.
If $|S_i| \ge 2$, we get the desired inequality by $\tau_Y(G) \ge k$. Thus we have $|S_i| \le 1$.
As $S_i\ne \emptyset$, we have $|S_i| = 1$.
If we assume by contradiction that  $k|S_{i}|>|B_1|+a_2+|C_i|+|T_i|+|A^*_{1i}|+e_{G}(S_i,U\setminus U_0) \ge |W|$,
then we know that  $|V(G_2) \cap Y| >2k+2-k$ as $|Y| \ge 2k+2$ by Claim~\ref{claim:Y-size}.
As the vertex of $S_i$ is contained in a component of $G_2$ with one single $Y$-vertex,
it follows that $c(G_2) \ge 2$.
However, $|W|/c(G_2) < k/2$, a contradiction to $\tau_Y(G)\ge k$.
\end{proof}

By Claims~\ref{claim1} and~\ref{claim2}, we get
\begin{eqnarray*}
&& k(b-|B_1|)=k(|S_1|+\ldots+|S_k|) \\
&\leq&2|C_1|+2|Z|+k|B_1|+ ka_2+\sum_{i=2}^k|C_i|+\sum\limits_{i=1}^{k-1}|T_i|+ \sum\limits_{i=2}^{k}|A_{1i}^*|+\sum\limits_{i=1}^{k}e_{G}(U\setminus U_0,S_i)-\ell \\
& \le &2|C_1|+2|Z|+k|B_1|+ka_2+\sum_{i=2}^k|C_i|+\sum\limits_{i=1}^{k-1}|T_i|+|A_1^*|+e_{G}(U\setminus U_0, B\setminus B_1)-\ell.
\end{eqnarray*}
On the other hand, as $\{A_{11}, Z, C_1, \ldots, C_k, A_1^*\}$ is a partition of $A_1$,
from~\eqref{eq6}, and using  $|A_{11}| \ge k|B_1|$ by~\eqref{eqn-R1b2}, we have
\begin{eqnarray*}
k(b-|B_1|)&>&ka_2+ 2a_1+e_{G}(U\setminus U_0,B)-\ell-k|B_1|\\
 &\ge & ka_2+2|A_{11}|+2|Z|+2\sum_{i=1}^k|C_i|+2|A_1^*|+e_{G}(U\setminus U_0,B\setminus B_1)-\ell-k|B_1| \\
 &\ge & ka_2+k|B_1| +2|Z|+2\sum_{i=1}^k|C_i|+2|A_1^*|
 +e_{G}(U\setminus U_0,B\setminus B_1)-\ell.
\end{eqnarray*}
Combining
the two sets of inequality above gives
\begin{equation}\label{kge3}
|A_1^*|+\sum_{i=2}^k|C_i|<	\sum\limits_{i=1}^{k-1}|T_i|.
\end{equation}

For each $i\in[1,k-1]$, let $|T_i|=t_i$.
By the maximality of $S_i$,
for each $y_j\in T_i$ with $j\in [1,t_i]$, there exists $x^j_i\in A^*_{1}$
such that $y_j\in N_{G}(x^j_i)$ and  $N_{G}(x^j_i)\cap B_2 \subseteq S_i\cup \{y_j\}$.
Thus we find
$t_i$ distinct vertices  $x_i^1, \ldots, x_i^{t_i}$
in $A_{1}^*$. Furthermore, for distinct $i,j\in [1,k-1]$, we
have $x_i^p \ne x_j^q$ by the definitions of $T_i, T_j$, where $p\in[1,t_i]$ and $q\in [1,t_j]$.
Therefore, we have $|A_1^*| \ge \sum_{i=1}^{k-1}|T_i|$.  This shows a contradiction
to~\eqref{kge3}.
 The proof of Theorem~\ref{thm2} is  now complete.

\section{Acknowledgement}

The authors wish to thank the anonymous reviewers for their valuable
suggestions, which greatly improved the presentation of this paper.

\bibliographystyle{abbrv}
\bibliography{arxiv}

\end{document}